\pdfoutput=1
\RequirePackage{ifpdf}
\ifpdf 
\documentclass[pdftex]{sigma}
\else
\documentclass{sigma}
\fi

\newtheorem{Theorem}{Theorem}[section]
\newtheorem{Corollary}[Theorem]{Corollary}
\newtheorem{Lemma}[Theorem]{Lemma}

 { \theoremstyle{definition}

\newtheorem{Example}[Theorem]{Example}
\newtheorem{Remark}[Theorem]{Remark} }

\numberwithin{equation}{section}

\newcommand{\ta}{\theta}
\newcommand{\C}{\mathbb C}

\begin{document}

\allowdisplaybreaks

\renewcommand{\thefootnote}{$\star$}

\newcommand{\arXivNumber}{1602.09027}

\renewcommand{\PaperNumber}{039}

\FirstPageHeading

\ShortArticleName{Elliptic Hypergeometric Summations by Taylor Series Expansion and Interpolation}

\ArticleName{Elliptic Hypergeometric Summations\\ by Taylor Series Expansion and Interpolation\footnote{This paper is a~contribution to the Special Issue
on Orthogonal Polynomials, Special Functions and Applications.
The full collection is available at \href{http://www.emis.de/journals/SIGMA/OPSFA2015.html}{http://www.emis.de/journals/SIGMA/OPSFA2015.html}}}

\Author{Michael J.~SCHLOSSER and Meesue YOO}
\AuthorNameForHeading{M.J.~Schlosser and M.~Yoo}
\Address{Fakult\"at f\"ur Mathematik, Universit\"at Wien,\\
Oskar-Morgenstern-Platz~1, A-1090 Vienna, Austria}
\Email{\href{mailto:michael.schlosser@univie.ac.at}{michael.schlosser@univie.ac.at},
\href{mailto:meesue.yoo@univie.ac.at}{meesue.yoo@univie.ac.at}}
\URLaddress{\url{http://www.mat.univie.ac.at/~schlosse/}}

\ArticleDates{Received March 01, 2016, in f\/inal form April 13, 2016; Published online April 19, 2016}

\Abstract{We use elliptic Taylor series expansions and interpolation
to deduce a number of summations for elliptic hypergeometric
series. We extend to the well-poised elliptic case results that in
the $q$-case have
previously been obtained by Cooper and by Ismail and Stanton.
We also provide identities involving S.~Bhargava's
cubic theta functions.}


\Classification{30E05; 33D15; 33D70; 33E05; 33E20}

\renewcommand{\thefootnote}{\arabic{footnote}}
\setcounter{footnote}{0}

\section{Introduction}\label{sec:intro}

Previously, one of us \cite{Schl} established an elliptic Taylor expansion
theorem which extends Ismail's~\cite{Is1} expansion
for functions symmetric in $z$ and $1/z$ in terms of the
Askey--Wilson monomial basis.
The expansion theorem in \cite{Schl} involves a special case
of Rains'~\cite{Ra2} elliptic extension of the Askey--Wilson
divided dif\/ference operator.
As applications, new simple proofs were given for
Frenkel and Turaev's~\cite{FT} elliptic extensions of Jackson's ${}_8\phi_7$
summation and of Bailey's ${}_{10}\phi_9$ transformation.
A further application concerned the computation of the connection
coef\/f\/icients of Spiridonov's~\cite{Sp} elliptic extension of Rahman's
biorthogonal rational functions.

Here we take a closer look at elliptic Taylor expansions.
In particular, we describe the action of the $m$-th elliptic
divided dif\/ference on a function, expressed in terms of the function.
In the ordinary case, if $\delta_h$ denotes the central dif\/ference
operator, def\/ined by $\delta_h\,f(x)=f(x+\frac h2)-f(x-\frac h2)$,
the $m$-th dif\/ference is given by
\begin{gather*}
\delta_h^m f(x)=\sum_{k=0}^m(-1)^k\binom mk
f\left(x+\left(\frac m2-k\right)h\right).
\end{gather*}
For the $q$-case, where $\delta_h$ is replaced
by the Askey--Wilson operator~$\mathcal D_q$,
acting on functions~$f(z)$ symmetric in $z$ and $1/z$,
an explicit formula for~$\mathcal D_q^mf(z)$ was
established by Cooper~\cite{Cooper}.
One of the results of our paper concerns
an extension of Cooper's formula to the elliptic setting.
We remark that Ismail, Rains and Stanton~\cite{IRS} independently
have also proved an elliptic extension of
Cooper's formula which turns out to be
equivalent to our result by a multiplication of operators.
In \cite{IS2}, Ismail and Stanton have used
Cooper's explicit formula to work out an explicit
interpolation formula for polynomials symmetric in
$z$ and $1/z$. Likewise, we use our elliptic extension
of Cooper's formula to f\/ind an elliptic interpolation formula.
Application of this formula yields single and multivariable
identities of Karlsson--Minton type.

Ismail and Stanton~\cite{IS} not only considered Taylor expansions
in terms of the Askey--Wilson monomial basis
$\{(az,a/z;q)_n,\, n\ge 0\}$
(see the subsequent subsection for the $q$-shifted factorial notation),
but also in terms of the basis
$\big\{\big(q^{\frac 14}z,q^{\frac 14}/z;q^{\frac 12}\big)_n,\, n\ge 0\big\}$,
for which they deduced quadratic summations as applications.
We are able to extend Ismail and Stanton's analysis and
provide, in particular, a Taylor expansion for an elliptic extension
of this other basis.
We note that in addition, Ismail and Stanton~\cite[Theorem~2.2]{IS} gave
a Taylor expansion theorem for the basis
$\big\{(1+z^2)\big({-}q^{2-n}z^2;q^2\big)_{n-1}z^{-n},\, n\ge 0\big\}$,
however this result (which involves an evaluation at $z=0$)
appears not to extend to the elliptic setting.

Finally, we consider series partially involving products of
S.~Bhargava's~\cite{Bhar} cubic theta functions.
Such series have not been considered before.
We introduce two dif\/ferent cubic theta extensions of
shifted factorials which are designed in such forms that they behave well
under the iterated action of the elliptic Askey--Wilson operator.
Applications of
Taylor expansion yield cubic theta extensions of
Jackson's~$_8\phi_7$ summation formula
and of a quadratic summation of Gessel and Stanton.

Before we present our new results, to make this paper
more self-contained, we brief\/ly review some
important material from the theory of elliptic hypergeometric series.
Afterwards we turn to the Askey--Wilson operator and
its elliptic extension, and then we provide our new results.

\subsection{Elliptic hypergeometric series}
For basic hypergeometric series, see Gasper and Rahman's
textbook \cite{GR}. Elliptic hypergeometric series
are treated there in Chapter~11.

By def\/inition, a function is {\em elliptic} if it is meromorphic and
doubly periodic. It is well known (cf., e.g.,~\cite{W}) that
elliptic functions can be built from quotients of
theta functions.

As building blocks we will use the \emph{modified Jacobi theta function}
with argument~$x$ and nome~$p$, def\/ined (in multiplicative notation) by
\begin{gather*}
\theta(x;p)= \prod_{j\ge 0}\big(\big(1-p^j x\big)\big(1- p^{j+1}/x\big)\big),\qquad
\theta(x_1,\dots, x_m;p)=\prod_{k=1}^m \theta(x_k;p),
\end{gather*}
where $x,x_1,\dots, x_m\ne 0$, $|p|<1$.

The modif\/ied Jacobi theta functions satisfy the following
basic properties which are essential in the theory of elliptic
hypergeometric series:
\begin{subequations}
\begin{gather}\label{tif}
\ta(x;p) =-x \ta(1/x;p),\\
\label{p1id}
\ta(px;p) =-\frac 1x \ta(x;p),
\end{gather}
and the {\em addition formula}
\begin{gather}\label{addf}
\ta(xy,x/y,uv,u/v;p)-\ta(xv,x/v,uy,u/y;p)
=\frac uy\,\ta(yv,y/v,xu,x/u;p)
\end{gather}
\end{subequations}
(cf.\ \cite[p.~451, Example~5]{WW}).

Note that in the theta function $\theta(x;p)$ we cannot let~$x\to 0$
(unless we f\/irst let $p\to 0$) for~$x$ is a pole of inf\/inite order.

Further, we def\/ine the {\em theta shifted factorial}
(or {\em $q,p$-shifted factorial}) by
\begin{gather*}
(a;q,p)_n = \begin{cases}
\displaystyle \prod^{n-1}_{k=0} \theta \big(aq^k;p\big),&\quad n = 1, 2, \ldots,\\
1,&\quad n = 0,\\
\displaystyle 1/\prod^{-n-1}_{k=0} \theta \big(aq^{n+k};p\big),&\quad n = -1, -2, \ldots,
\end{cases}
\end{gather*}
together with
\begin{gather*}
(a_1, a_2, \ldots, a_m;q, p)_n = \prod^m_{k=1} (a_k;q,p)_n,
\end{gather*}
for compact notation.
For $p=0$ we have $\theta (x;0) = 1-x$ and, hence, $(a;q, 0)_n = (a;q)_n$
is a~{\em $q$-shifted factorial} in base $q$.
The parameters $q$ and $p$
in $(a;q,p)_n$ are called the {\em base} and {\em nome}, respectively.
Observe that
\begin{gather*}
(pa;q,p)_n=(-1)^na^{-n}q^{-\binom n2} (a;q,p)_n,
\end{gather*}
which follows from~\eqref{p1id}.
A list of other useful identities for manipulating the
$q,p$-shifted factorials is given in~\cite[Section~11.2]{GR}.

A series $\sum c_n$ is called an {\it elliptic hypergeometric series} if
$g(n) = c_{n+1}/c_n$ is an elliptic function of $n$ with $n$
considered as a complex variable, i.e., the function~$g(x)$
is a~doubly periodic meromorphic function of the complex variable~$x$.
Without loss of generality, by the theory of theta functions,
one may assume that
\begin{gather*}
g(x)=\frac{\ta\big(a_1q^x,a_2q^x,\dots,a_{s+1}q^x;p\big)}
{\ta\big(q^{1+x},b_1q^x,\dots,b_sq^x;p\big)} z,
\end{gather*}
where the {\em elliptic balancing condition}, namely
\begin{gather*}
a_1a_2\cdots a_{s+1}=qb_1b_2\cdots b_s,
\end{gather*}
holds.
If we write $q=e^{2\pi i\sigma}$, $p=e^{2\pi i\tau}$,
with complex $\sigma$, $\tau$, then $g(x)$ is indeed periodic in~$x$
with periods $\sigma^{-1}$ and $\tau\sigma^{-1}$.

For convergence reasons, one usually requires $a_{s+1}=q^{-n}$
($n$ being a nonnegative integer),
so that the sum of an elliptic hypergeometric series
is in fact f\/inite.

{\em Very-well-poised elliptic hypergeometric series} are def\/ined as
\begin{gather*}
{}_{s+1}V_s(a_1;a_6,\dots,a_{s+1};q,p):=
\sum_{k=0}^{\infty}\frac{\ta\big(a_1q^{2k};p\big)}{\ta(a_1;p)}
\frac{(a_1,a_6,\dots,a_{s+1};q,p)_k}
{(q,a_1q/a_6,\dots,a_1q/a_{s+1};q,p)_k}(qz)^k,
\end{gather*}
where
\begin{gather*}
q^2a_6^2a_7^2\cdots a_{s+1}^2=(a_1q)^{s-5}.
\end{gather*}
Note that in the elliptic case the number of
pairs of numerator and denominator parameters
involved in the construction of the {\em very-well-poised term}
$\theta(a_1q^{2k};p)/\theta(a_1;p)$
is {\em four}
(whereas in the basic case this number is {\em two},
in the ordinary case only {\em one}). See Spiridonov~\cite{Sp} or
Gasper and Rahman~\cite[Chapter~11]{GR} for details.

In their study of elliptic $6j$ symbols (which are elliptic solutions
of the Yang--Baxter equation found by Baxter~\cite{B} and Date et
al.~\cite{DJKMO}), Frenkel and Turaev~\cite{FT}
discovered the following $_{10}V_9$ summation formula
(as a result of a more general
${}_{12}V_{11}$ transformation, being a consequence
of the tetrahedral symmetry of the elliptic $6j$ symbols):
\begin{gather}\label{10V9}
{}_{10}V_9\big(a;b,c,d,e,q^{-n};q,p\big)
=\frac {(aq,aq/bc,aq/bd,aq/cd;q,p)_n}
{(aq/b,aq/c,aq/d,aq/bcd;q,p)_n},
\end{gather}
where $a^2q^{n+1}=bcde$.
The $_{10}V_9$ summation is an elliptic analogue of Jackson's
$_8\phi_7$ summation formula (cf.\ \cite[equation~(2.6.2)]{GR})
\begin{gather}\label{8phi7}
 \sum_{k=0}^n\frac{(1-aq^{2k})(a,b,c,d,e,q^{-n};q)_k}
{(1-a)(q,aq/b,aq/c,aq/d,aq/e,aq^{n+1};q)_k}q^k
=\frac {(aq,aq/bc,aq/bd,aq/cd;q)_n}
{(aq/b,aq/c,aq/d,aq/bcd;q)_n},
\end{gather}
where $a^2q^{n+1}=bcde$, which in turn is a $q$-analogue
of Dougall's $_7F_6$ summation formula.

\subsection{The Askey--Wilson operator}

The Askey--Wilson operator $\mathcal D_q$ was f\/irst def\/ined in~\cite{AW}.
We consider meromorphic functions~$f(z)$ symmetric in~$z$
and~$1/z$. Writing $z=e^{i\theta}$ (note that $\theta$ need not to be real),
we may consider~$f$ to be a function in $x=\cos\theta=(z+1/z)/2$ and write
$f[x]:=f(z)$. (I.e., $f$~can be considered as a function in $z$, or
equivalently, as a function in~$x$, where the two dif\/ferent notations
specify the dependency to be considered.)

The {\em Askey--Wilson operator} acts on functions of
$x=\cos\theta$. It is def\/ined as follows:
\begin{gather}\label{awop}
\mathcal D_q f[x]=\frac{f\big(q^{\frac 12}z\big)-f\big(q^{-\frac 12}z\big)}
{\iota\big(q^{\frac 12}z\big)-\iota\big(q^{-\frac 12}z\big)},
\end{gather}
where $\iota[x]=x$ (i.e., $\iota(z)=(z+1/z)/2$).
Equation~\eqref{awop} can also be written as
\begin{gather*}
\mathcal D_q f[x]=\frac{f\big(q^{\frac 12}z\big)-f\big(q^{-\frac 12}z\big)}
{i\big(q^{\frac 12}-q^{-\frac 12}\big)\sin\theta}.
\end{gather*}
The operator $\mathcal D_q$ is a
$q$-analogue of the dif\/ferentiation operator
(which is dif\/ferent to Jackson's $q$-dif\/ference operator).
In particular, since
\begin{gather*}
\mathcal D_q T_n[x]=
\frac{q^{\frac n2}-q^{-\frac n2}}
{q^{\frac 12}-q^{-\frac 12}}U_{n-1}[x],
\end{gather*}
where $T_n[\cos\theta]=\cos n\theta$ and
$U_n[\cos\theta]=\sin(n+1)\theta/\sin\theta$ are the Chebyshev polynomials
of the f\/irst and second kind, one easily sees that $\mathcal D_q$ maps
polynomials to polynomials, lowering the degree by one.

In the calculus of the Askey--Wilson operator the so-called
``Askey--Wilson monomials''\break $\phi_n(x;a)=(az,a/z;q)_n$
form a natural basis for polynomials or power series in $x$.
One readily computes
\begin{gather*}
\mathcal D_q (az,a/z;q)_n=-\frac{2a (1-q^n)}{(1-q)}
\big(aq^{\frac 12}z,aq^{\frac 12}/z;q\big)_{n-1}.
\end{gather*}
Ismail~\cite{Is1} proved the following Taylor theorem for polynomials $f[x]$.
\begin{Theorem}\label{thism}
If $f[x]$ is a polynomial in $x$ of degree~$n$, then
\begin{gather*}
f[x]=\sum_{k=0}^nf_k\phi_k(x;a),
\end{gather*}
where
\begin{gather*}
f_k=\frac{(q-1)^k}{(2a)^k(q;q)_k}q^{-k(k-1)/4}
\big[\mathcal D_q^k f[x]\big]_{x=x_k}, \qquad
x_k:=\frac 12\big(aq^{\frac k2}+q^{-\frac k2}/a\big).
\end{gather*}
\end{Theorem}

As it was shown in \cite{Is1},
the application of Theorem~\ref{thism} to $f(z)=(bz,b/z;q)_n$
immediately gives the $q$-Pfaf\/f--Saalsch\"utz summation
(cf.\ \cite[equation~(1.7.2)]{GR}), in the form
\begin{gather*}
\frac{(bz,b/z;q)_n}{(ba,b/a;q)_n}=
{}_3\phi_2\!\left[\begin{matrix}az,a/z,q^{-n}\\
ab,q^{1-n}a/b\end{matrix};q,q\right],
\end{gather*}
and its application to
the Askey--Wilson polynomials,
\begin{gather*}
\omega_n(x;a,b,c,d;q):=
{}_4\phi_3\!\left[\begin{matrix}az,a/z,abcdq^{n-1},q^{-n}\\
ab,ac,ad\end{matrix};q,q\right],
\end{gather*}
gives a connection coef\/f\/icient identity which,
by specialization, can be reduced to the Sears transformation
(cf.\ \cite[equation~(3.2.1)]{GR}), in the form
\begin{gather*}
\omega_n(x;a,b,c,d;q)=
\frac{a^n(bc,bd;q)_n}{b^n(ac,ad;q)_n}\,
\omega_n(x;b,a,c,d;q).
\end{gather*}

Ismail and Stanton~\cite{IS} extended the above polynomial Taylor theorem
to hold for entire functions of exponential growth,
resulting in inf\/inite Taylor expansions.
Marco and Parcet~\cite{MP} extended this yet further to hold for
arbitrary $q$-dif\/ferentiable functions, resulting in
inf\/inite Taylor expansions with explicit remainder term.
Among other results they were able to
recover the nonterminating $q$-Pfaf\/f--Saalsch\"utz summation
(cf.\ \cite[Appendix~(II.24)]{GR}).

\subsection{The well-poised and elliptic Askey--Wilson operator}
Since
\begin{gather*}
\mathcal D_q \frac{(az,a/z;q)_n}{(cz,c/z;q)_n}
=\frac 2{\big(q^{\frac 12}-q^{-\frac 12}\big)(z-1/z)}
\left[\frac{\big(aq^{\frac 12}z,aq^{-\frac 12}/z;q\big)_n}
{\big(cq^{\frac 12}z,cq^{-\frac 12}/z;q\big)_n}-
\frac{\big(aq^{-\frac 12}z,aq^{\frac 12}/z;q\big)_n}
{\big(cq^{-\frac 12}z,cq^{\frac 12}/z;q\big)_n}\right]\\
\qquad{}
=
\frac 2{\big(q^{\frac 12}-q^{-\frac 12}\big)(z-1/z)}
\frac{\big(aq^{\frac 12}z,aq^{\frac 12}/z;q\big)_{n-1}}
{\big(cq^{\frac 12}z,cq^{\frac 12}/z;q\big)_{n-1}}\\
\qquad\quad{}
\times
\left[\frac{\big(1-azq^{n-\frac 12}\big)\big(1-aq^{-\frac 12}/z\big)}
{\big(1-czq^{n-\frac 12}\big)\big(1-cq^{-\frac 12}/z\big)}-
\frac{\big(1-azq^{-\frac 12}\big)\big(1-aq^{n-\frac 12}/z\big)}
{\big(1-czq^{-\frac 12}\big)\big(1-cq^{n-\frac 12}/z\big)}
\right]\\
\qquad{}
=
\frac{(-1)2a(1-c/a)\big(1-acq^{n-1}\big)(1-q^n)}
{\big(1-czq^{-\frac 12}\big)\big(1-czq^{\frac 12}\big)
\big(1-cq^{-\frac 12}/z\big)\big(1-cq^{\frac 12}/z\big)(1-q)}
\frac{\big(aq^{\frac 12}z,aq^{\frac 12}/z;q\big)_{n-1}}
{\big(cq^{\frac 32}z,cq^{\frac 32}/z;q\big)_{n-1}},
\end{gather*}
we were led in \cite{Schl} to def\/ine a
$c$-generalized well-poised Askey--Wilson operator
acting on~$x$ (or~$z$) by
\begin{gather*}
\mathcal D_{c,q}=\big(1-czq^{-\frac 12}\big)\big(1-czq^{\frac 12}\big)
\big(1-cq^{-\frac 12}/z\big)\big(1-cq^{\frac 12}/z\big) \mathcal D_q,
\end{gather*}
which acts ``degree-lowering'' on the ``rational monomials''
(or ``well-poised monomials'')
\begin{gather*}
\frac{(az,a/z;q)_n}{(cz,c/z;q)_n}
\end{gather*}
in the form
\begin{gather*}
\mathcal D_{c,q} \frac{(az,a/z;q)_n}{(cz,c/z;q)_n}=
\frac{(-1)2a(1-c/a)\big(1-acq^{n-1}\big)(1-q^n)}{(1-q)}
\frac{\big(aq^{\frac 12}z,aq^{\frac 12}/z;q\big)_{n-1}}
{\big(cq^{\frac 32}z,cq^{\frac 32}/z;q\big)_{n-1}}.
\end{gather*}
Clearly, $\mathcal D_{0,q}=\mathcal D_{q}$.

More generally, for parameters $c$, $q$, $p$ with $|q|,|p|<1$, we def\/ined an
elliptic extension of the Askey--Wilson operator, acting on
functions symmetric in $z^{\pm1}$, by
\begin{gather}\label{defdf}
\mathcal D_{c,q,p} f(z)
=2q^{\frac 12}z \frac{\theta\big(czq^{-\frac 12},czq^{\frac 12},
cq^{-\frac 12}/z,cq^{\frac 12}/z;p\big)}{\theta(q,z^2;p)}
\big(f\big(q^{\frac 12}z\big)-f\big(q^{-\frac 12}z\big)\big).
\end{gather}
Note that $\mathcal D_{c,q,0}=\mathcal D_{c,q}$.

In particular, using \eqref{addf}, we have
\begin{gather}\label{degl}
\mathcal D_{c,q,p} \frac{(az,a/z;q,p)_n}{(cz,c/z;q,p)_n}=
\frac{(-1)2a \theta\big(c/a,acq^{n-1},q^n;p\big)}{\theta(q;p)}
\frac{\big(aq^{\frac 12}z,aq^{\frac 12}/z;q,p\big)_{n-1}}
{\big(cq^{\frac 32}z,cq^{\frac 32}/z;q,p\big)_{n-1}}.
\end{gather}

\begin{Remark}
The operator $\mathcal D_{c,q,p}$ happens to be a special case
of a multivariable dif\/ference operator introduced by Rains in~\cite{Ra1}.
Already in the single variable case Rains' operator
involves two more parameters than $\mathcal D_{c,q,p}$.
(Rains' dif\/ference operators generate a representation of
the Sklyanin algebra, as observed in~\cite{Ra1} and made
explicit in~\cite{R1} and \cite[Section~6]{R2}.)
Rains' operator can be specialized to act as degree-lowering
(as the above $\mathcal D_{c,q,p}$ does), degree-preserving
or degree-raising on abelian functions.
Rains used his multivariable dif\/ference operators in \cite{Ra1}
to construct $BC_n$-symmetric biorthogonal abelian functions
which generalize Koornwinder's orthogonal polynomials.
He further used his operator in~\cite{Ra2} to derive $BC_n$-symmetric
extensions of Frenkel and Turaev's~${}_{10}V_9$ summation
and~${}_{12}V_{11}$ transformation.
\end{Remark}

\section{Elliptic Taylor expansions and interpolation}
We work in the following space of abelian functions.

For a complex number $c$, let
\begin{gather*}
W_c ^{n} := \text{span}_\C\left\{\frac{g_k (z)}{(cz, c/z;q,p)_k},\,
0\le k\le n \right\},
\end{gather*}
where $g_k(z)$ runs over all functions being holomorphic for
$z\ne 0$ with $g_k (z)=g_k(1/z)$ and
\begin{gather*}
g_k(pz)=\frac{1}{p^k z^{2k}}g_k (z).
\end{gather*}

In classical terminology, $g_k(z)$ is an even theta function
of order $2k$ and zero characteristic. Rains~\cite{Ra2}
refers to such functions as $BC_1$ theta functions of degree $k$,
whereas in Rosengren and Schlosser~\cite{RS} they are referred to
as $D_k$ theta functions.
It is well-known that the space~$V^k$ of even theta functions
of order~$2k$ and zero characteristic has dimension~$k+1$
(see, e.g., Weber~\cite[p.~49]{W}).

Note that $W_c^n$ consists of certain abelian functions.
(For $p\to 0$ these degenerate to certain rational functions
which we may call ``well-poised''.)
\begin{Lemma}[\protect{\cite[Lemma~4.1]{Schl}}]\label{Lem}
For any arbitrary but fixed complex number $a$
$($satisfying $a\neq c q^jp^k$, for $j=0,\dots,n-1$,
and $k\in\mathbb Z$, and $a\neq q^jp^k/c$,
for $j=2-2n,\dots,1-n$, and $k\in\mathbb Z)$, the set
\begin{gather*}
\left\{\frac{(az,a/z;q,p)_k}
{(cz,c/z;q,p)_k},\, 0\le k\le n\right\}
\end{gather*}
forms a basis for $W_c^n$.
\end{Lemma}

Note that, in view of \eqref{degl},
the elliptic Askey--Wilson operator maps functions in $W_c^n$ to
functions in $W_{cq^{\frac 32}}^{n-1}$.

We now def\/ine
\begin{gather*}
\mathcal D_{c,q,p}^{(k)}=\mathcal D_{cq^{\frac 32},q,p}^{(k-1)}
\mathcal D_{c,q,p},
\end{gather*}
with $\mathcal D_{c,q,p}^{(0)}=\varepsilon$, the identity operator.
We have the following elliptic Taylor expansion theorem which extends
Theorem~\ref{thism} of Ismail.
\begin{Theorem}[\protect{\cite[Theorem 4.2]{Schl}}]\label{thm:Taylor}
If $f$ is in $W_c^n$, then
\begin{gather*}
f(z)=\sum_{k=0}^nf_k \frac{(az,a/z;q,p)_k}
{(cz,c/z;q,p)_k},
\end{gather*}
where
\begin{gather*}
f_k=\frac{(-1)^kq^{-k(k-1)/4} \ta(q;p)^k}{(2a)^k(q,c/a,acq^{k-1};q,p)_k}
\big[\mathcal D_{c,q,p}^{(k)} f(z)\big]_{z=aq^{\frac k2}}.
\end{gather*}
\end{Theorem}

\begin{Example}
Let
\begin{gather*}
f(z)=\frac{(bz,b/z;q,p)_n}
{(cz,c/z;q,p)_n}.
\end{gather*}
Application of Theorem~\ref{thm:Taylor} in conjunction with \eqref{degl} gives
\begin{gather*}
f_k =\frac{(-1)^kq^{-k(k-1)/4}
\ta(q;p)^k}{(2a)^k(q,c/a,acq^{k-1};q,p)_k}\\
\hphantom{f_k =}{}\times
(-1)^k(2b)^kq^{k(k-1)/4}
\frac{(q;q,p)_n\big(c/b,bcq^{n-1};q,p\big)_k}{(q;q,p)_{n-k} \ta(q;p)^k}
\frac{\big(abq^k,b/a;q,p\big)_{n-k}}{\big(acq^{2k},cq^k/a;q,p\big)_{n-k}}\\
\hphantom{f_k}{}=
\frac{(ab,b/a;q,p)_n}{(ac,c/a;q,p)_n}\frac{\ta\big(acq^{2k-1};p\big)}{\ta\big(acq^{-1};p\big)}
\frac{\big(acq^{-1},c/b,bcq^{n-1},q^{-n};q,p\big)_k}{\big(q,ab,aq^{1-n}/b,acq^n;q,p\big)_k}q^k,
\end{gather*}
thus yielding Frenkel and Turaev's ${}_{10}V_9$ summation \eqref{10V9},
in the form
\begin{gather*}
\frac{(ac,c/a,bz,b/z;q,p)_n}
{(ab,b/a,cz,c/z;q,p)_n}={}_{10}V_9\big(acq^{-1};az,a/z,c/b,bcq^{n-1},q^{-n};q,p\big).
\end{gather*}
\end{Example}

We now prove an elliptic extension
of a theorem of S.~Cooper~\cite{Cooper} which explicitly
describes the action of the $m$-iterated Askey--Wilson operator.

\begin{Theorem}\label{thm:Dn}
The action of $\mathcal{D}_{c,q,p}^{(m)}$ on a function $f\in W_c^n$
is given by
\begin{gather}\label{eqn:Dn}
\mathcal{D}_{c,q,p}^{(m)}f(z) = (-2z)^mq^{\frac{m(3-m)}{4}}
\frac{\big(c q^{\frac{m}{2}-1}z, c q^{\frac{m}{2}-1}/z;q,p\big)_{m+1}}
{(\theta(q;p))^m} \\
\hphantom{\mathcal{D}_{c,q,p}^{(m)}f(z) = }{}
\times\sum_{k=0}^{m} q^{k(m-k)}
\begin{bmatrix}m\\k\end{bmatrix}_{p,q}
\frac{z^{2(k-m)}\big(c q^{\frac{m}{2}-k}z, cq^{-\frac{m}{2}+k}/z;q,p\big)_{m-1}}
{\big(q^{m-2k+1}z^2 ;q,p\big)_k
\big(q^{2k-m+1}z^{-2};q,p\big)_{m-k}} f\big(q^{\frac{m}{2}-k}z\big),\nonumber
\end{gather}
where
\begin{gather*}
\begin{bmatrix}m\\k\end{bmatrix}_{p,q}
=\frac{\big(q^{1+k};q,p\big)_{m-k}}{(q;q,p)_{m-k}}.
\end{gather*}
\end{Theorem}

\begin{proof}
We prove this by induction. If $m=1$, then \eqref{eqn:Dn} just
reduces to the def\/inition of $\mathcal{D}_{c,q,p}f(z)$ in~\eqref{defdf}.
Now say~\eqref{eqn:Dn} holds up to some~$m$.
Then if we let $f^{(m)}(z):= \mathcal{D}_{c,q,p}^{(m)}f(z)$,
\begin{gather*}
 \mathcal{D}_{c,q,p}^{(m+1)}f(z) =
\mathcal{D}_{cq^{\frac{3}{2}m},q,p}f^{(m)}(z)\\
 {} = 2q^{\frac{1}{2}}z\frac{\big(czq^{\frac{3}{2}m-\frac{1}{2}},
cq^{\frac{3}{2}m-\frac{1}{2}}/z;q,p\big)_2}{\theta(q,z^2;p)}
\big(f^{(m)}\big(q^{\frac{1}{2}}z\big)-f^{(m)}\big(q^{-\frac{1}{2}}z\big) \big)\\
 {} = 2q^{\frac{1}{2}}z\frac{\big(czq^{\frac{3}{2}m-\frac{1}{2}},
cq^{\frac{3}{2}m-\frac{1}{2}}/z;q,p\big)_2}{\theta(q,z^2;p)}
\frac{(-2z)^m q^{\frac{m(3-m)}{4}}}{(\theta(q;p))^m}
\Bigg\{ q^{\frac{m}{2}} \big(cq^{\frac{m}{2}-\frac{1}{2}}z,
cq^{\frac{m}{2}-\frac{3}{2}}/z;q,p\big)_{m+1} \\
 \quad {} \times \sum_{k=0}^{m} q^{(k-1)(m-k)}
\begin{bmatrix}m\\k\end{bmatrix}_{p,q}\frac{z^{2(k-m)}
\big(c q^{\frac{m}{2}-k+\frac{1}{2}}z, cq^{-\frac{m}{2}+k-\frac{1}{2}}/z;q,p\big)_{m-1}}
{\big(q^{m-2k+2}z^2 ;q,p\big)_k \big(q^{2k-m}z^{-2};q,p\big)_{m-k}} f\big(q^{\frac{m}{2}-k+\frac{1}{2}}z\big)\\
 \quad {}-q^{-\frac{m}{2}} \big(cq^{\frac{m}{2}-\frac{3}{2}}z,
cq^{\frac{m}{2}-\frac{1}{2}}/z;q,p\big)_{m+1}\\
 \quad{} \times \sum_{k=0}^{m} q^{(k+1)(m-k)}
\begin{bmatrix}m\\k\end{bmatrix}_{p,q}\frac{z^{2(k-m)}
\big(c q^{\frac{m}{2}-k-\frac{1}{2}}z, cq^{-\frac{m}{2}+k+\frac{1}{2}}/z;q,p\big)_{m-1}}
{\big(q^{m-2k}z^2 ;q,p\big)_k \big(q^{2k-m+2}z^{-2};q,p\big)_{m-k}}
f\big(q^{\frac{m}{2}-k-\frac{1}{2}}z\big)\Bigg\}\\
{}= (2z)^{m+1}q^{\frac{m(3-m)}{4}+\frac{m+1}{2}}(-1)^m
\frac{\big(czq^{\frac{3}{2}m-\frac{1}{2}}, cq^{\frac{3}{2}m-\frac{1}{2}}/z;q,p\big)_2}
{(\theta(q;p))^{m+1}\theta(z^2;p)} \big(c q^{\frac{m}{2}-\frac{1}{2}}z,
cq^{\frac{m}{2}-\frac{1}{2}}/z;q,p\big)_m\\
\quad{} \times\left\{ \theta\big(cq^{\frac{3}{2}m-\frac{1}{2}}z,
cq^{\frac{m}{2}-\frac{3}{2}}/z;p\big)\right.\\
 \quad{} \times \sum_{k=0}^{m} q^{(k-1)(m-k)}
\begin{bmatrix}m\\k\end{bmatrix}_{p,q}\frac{z^{2(k-m)}
\big(c q^{\frac{m}{2}-k+\frac{1}{2}}z, cq^{-\frac{m}{2}+k-\frac{1}{2}}/z;q,p\big)_{m-1}}
{\big(q^{m-2k+2}z^2 ;q,p\big)_k \big(q^{2k-m}z^{-2};q,p\big)_{m-k}} f\big(q^{\frac{m}{2}-k+\frac{1}{2}}z\big)\\
 \quad {} -q^{-m} \theta\big(cq^{\frac{m}{2}-\frac{3}{2}}z,
cq^{\frac{3}{2}m-\frac{1}{2}}/z;q,p\big)\\
\left. \quad{}\times \sum_{k=1}^{m+1} q^{k(m-k+1)}\!
\begin{bmatrix}m\\k-1\end{bmatrix}_{p,q}\!\frac{z^{2(k-m-1)}
\big(c q^{\frac{m}{2}-k+\frac{1}{2}}z, cq^{-\frac{m}{2}+k-\frac{1}{2}}/z;q,p\big)_{m-1}}
{\big(q^{m-2k+2}z^2 ;q,p\big)_{k-1} \big(q^{2k-m}z^{-2};q,p\big)_{m-k+1}}
f\big(q^{\frac{m}{2}-k+\frac{1}{2}}z\big)\right\}\\
 = (2z)^{m+1}q^{\frac{m(3-m)}{4}+\frac{m+1}{2}}(-1)^m
\frac{\big(c q^{\frac{m}{2}-\frac{1}{2}}z,
cq^{\frac{m}{2}-\frac{1}{2}}/z;q,p\big)_{m+2}}{(\theta(q;p))^{m+1}\theta\big(z^2;p\big)}\\
 \quad {}\times \left\{ \theta\big(cq^{\frac{3}{2}m-\frac{1}{2}}z,
cq^{\frac{m}{2}-\frac{3}{2}}/z;p\big) \times \frac{q^{-m}z^{-2m}
\big(cq^{\frac{m}{2}+\frac{1}{2}}z, cq^{-\frac{m}{2}-\frac{1}{2}}/z;q,p\big)_{m-1}}
{\big(q^{-m}z^{-2};q,p\big)_n}f\big(q^{\frac{m}{2}+\frac{1}{2}}z\big)\right.\\
 \quad{} -q^{-m} \theta\big(cq^{\frac{m}{2}-\frac{3}{2}}z,
cq^{\frac{3}{2}m-\frac{1}{2}}/z;p\big) \times \frac{\big(cq^{-\frac{m}{2}-\frac{1}{2}}z,
cq^{\frac{m}{2}+\frac{1}{2}}/z;q,p\big)_{m-1}}{\big(q^{-m}z^{2};q,p\big)_m}
f\big(q^{-\frac{m}{2}-\frac{1}{2}}z\big)\\
\quad{} + q^{-m}\sum_{k=1}^m q^{k(m-k+1)}\frac{\big(q^{1+k};q,p\big)_{m-k}}
{(q;q,p)_{m-k+1}}\frac{\big(cq^{\frac{m}{2}-k+\frac{1}{2}}z,
cq^{-\frac{m}{2}+k-\frac{1}{2}}/z;q,p\big)_{m-1}f\big(q^{\frac{m}{2}-k+\frac{1}{2}}z\big)}
{\big(q^{m-2k+2}z^2 ;q,p\big)_k \big(q^{2k-m}z^{-2};q,p\big)_{m-k+1}}\\
 \quad{} \times z^{2k-2m}\left[ \theta\big(q^{m-k+1},
c q^{\frac{3}{2}m -\frac{1}{2}}z, c q^{\frac{m}{2}-\frac{3}{2}}/z, q^k z^{-2};p\big)
\right.\\
\left.\left. \quad{}
-z^{-2}\theta\big(q^{k}, c q^{\frac{m}{2} -\frac{3}{2}}z, c q^{\frac{3}{2}m-\frac{1}{2}}/z,
q^{m-k+1} z^{2};p\big)\right]\right\}\\
{} = (-2z)^{m+1}q^{\frac{(m+1)(2-m)}{4}}
\frac{\big(c q^{\frac{m}{2}-\frac{1}{2}}z,
c q^{\frac{m}{2}-\frac{1}{2}}/z;q,p\big)_{m+2}}{(\theta(q;p))^{m+1}}\\
 \quad{} \times \sum_{k=0}^{m+1} q^{k(m+1-k)}
\begin{bmatrix}m+1\\k\end{bmatrix}_{p,q}\frac{z^{2(k-m-1)}
\big(c q^{\frac{m+1}{2}-k}z, cq^{-\frac{m+1}{2}+k}/z;q,p\big)_{m}}
{\big(q^{m-2k+2}z^2 ;q,p\big)_k \big(q^{2k-m}z^{-2};q,p\big)_{m+1-k}} f\big(q^{\frac{m+1}{2}-k}z\big).
\end{gather*}
Hence the theorem is proved. Notice that in the last step we used the
addition formula \eqref{addf} with the substitutions
\begin{gather*}
(x,y,u,v)\mapsto\big(cq^{m-1},q^{\frac{m+1}2}z,q^{\frac{m+1}2}/z,q^{\frac{m+1}2-k}z\big)
\end{gather*}
to simplify the summand.
\end{proof}

\begin{Remark}
Our formula in Theorem~\ref{thm:Dn} involves
the $c$-dependent elliptic Askey--Wilson ope\-rator ${\cal D}^{(m)}_{c,q,p}$.
Independently, Ismail, Rains and Stanton~\cite[Proposition~8.1]{IRS}
also gave an elliptic extension of Cooper's result.
Although at f\/irst glance Ismail, Rains and Stanton's result,
which involves a slightly dif\/ferent operator
(without the denominator variable $c$),
looks dif\/ferent from ours in Theorem~\ref{thm:Dn},
the two results are indeed equivalent up to a
multiplication of operators. (We would like to thank Eric Rains
for helping us to clarify this.) In fact, the operator
\begin{gather*}
\big(q^{\frac m2-1}c z,q^{\frac m2-1}c/z;q,p\big)_{m+1}^{-1}
{\cal D}^{(m)}_{c,q,p}
(cz,c/z;q,p)_{m-1}^{-1}
\end{gather*}
(which acts on functions $h(z)=g(z)/(czq^{m-1},cq^{m-1}/z;q,p)_{n-m+1}$
with $g(z)=g(1/z)$ and $g(pz)=p^{-n}z^{-2n}g(z)$)
is independent of $c$, thus our operator ${\cal D}^{(m)}_{c,q,p}$
is proportional to{\samepage
\begin{gather*}
 v^{2m} \theta\big(q^{\frac m2}v z,q^{\frac m2}v/z;q\big)
\big(q^{\frac m2 -1}c z,q^{\frac m2-1}c/z;q,p\big)_{m+1}
{\cal D}_m(q;p)
(cz,c/z;q,p)_{m-1}
\theta(vz,v/z;q)^{-1}\!,
\end{gather*}
with ${\cal D}_m(q;p)$ from \cite[Section~8]{IRS}, where
the constant of proportionality is independent of~$v$.}
\end{Remark}

We are now able to obtain an elliptic extension of Ismail and
Stanton's~\cite[Theorem~3.4]{IS2}
interpolation formula. In particular, for any $a\in\C$, the function
$f\in W_c^n$ is uniquely determined by its evaluation at
the $n+1$ interpolation
points $a,aq,\dots,aq^n$, with closed form coef\/f\/icients.

\begin{Theorem}\label{thm:id}
If $f$ is in $W_c^n$, then
\begin{gather*}
\frac{\big(a^2 q, q, cz, c/z;q,p\big)_n}{(ac, c/a, aqz, aq/z;q,p)_n}f(z)\\
\qquad{}
= \sum_{k=0}^n q^k \frac{\theta\big(a^2 q^{2k};p\big)}{\theta(a^2 ;p)}
\frac{\big( q^{-n}, a^2 ,aq/c,ac q^n, az, a/z;q,p\big)_k}
{\big(q, a^2 q^{n+1},ac, aq^{1-n}/c, aqz, aq/z;q,p\big)_k}f\big(a q^k\big).
\end{gather*}
\end{Theorem}

\begin{proof}
By combining Theorem \ref{thm:Taylor} and Theorem \ref{thm:Dn}, we f\/ind that
\begin{gather*}
 f(z)= \sum_{k=0}^n q^k \theta\big(c q^{-1}/a, ac q^{2k-1};p\big)
\frac{(az, a/z;q,p)_k}{(cz, c/z;q,p)_k}\notag\\
\hphantom{f(z)=}{} \times \sum_{j=0}^k q^{-(k-j)^2}a^{2(j-k)}
\frac{\big(ac q^{k-j}, cq^{-k+j}/a;q,p\big)_{k-1}}
{\big(q,a^2 q^{2k-2j+1};q,p\big)_j \big(q, a^{-2}q^{2j -2k+1};q,p\big)_{k-j}}f\big(aq^{k-j}\big).
\end{gather*}
By shifting the index $k\mapsto k+j$, we get
\begin{gather*}
 f(z)
 =\sum_{k=0}^n \sum_{j=0}^{n-k} q^{k+j} \theta\big(c q^{-1}/a, ac q^{2k+2j-1};p\big)
\frac{(az, a/z;q,p)_{k+j}}{(cz, c/z;q,p)_{k+j}}\\
\hphantom{f(z)=}{} \times q^{-k^2}a^{-2k}\frac{\big(ac q^{k}, cq^{-k}/a;q,p\big)_{k+j-1}}
{\big(q,a^2 q^{2k+1};q,p\big)_j \big(q, a^{-2}q^{-2k+1};q,p\big)_{k}}f\big(aq^{k}\big)\\
\hphantom{f(z)}{}
 = \sum_{k=0}^n q^{k(1-k)}a^{-2k} \frac{\big(ac q^{k}, cq^{-k}/a, az, a/z;q,p\big)_k}
{\big(q, a^{-2}q^{-2k +1}, cz, c/z ;q,p\big)_k}f\big(a q^{k}\big)\\
\hphantom{f(z)=}{}
 \times \sum_{j=0}^{n-k}q^j \frac{\theta\big(ac q^{2k+2j-1};p\big)}
{\theta\big(ac q^{2k-1};p\big)}
\frac{\big(ac q^{2k-1}, cq^{-1}/a, aq^k z, aq^k /z, q^{-n+k}, ac q^{n+k};q,p\big)_j}
{\big(q, a^2 q^{2k+1}, c q^{k} /z, cq^k z, q^{-n+k}, ac q^{n+k};q,p\big)_j}\\
\hphantom{f(z)}{}
= \sum_{k=0}^n q^{k(1-k)}a^{-2k} \\
\hphantom{f(z)=}{}\times
\frac{\big(ac q^{k}, cq^{-k}/a, az, a/z;q,p\big)_k
\big(ac q^{2k}, aq^{k+1}/z, aq^{k+1}z, c/a;q,p\big)_{n-k}}
{\big(q, a^{-2}q^{-2k +1}, cz, c/z ;q,p\big)_k
\big(q, a^2 q^{2k+1}, c q^k /z, c q^k z;q,p\big)_{n-k}}f\big(a q^{k}\big).
\end{gather*}
The last sum was obtained by virtue of the Frenkel and Turaev
summation formula~\eqref{10V9}. The theorem then follows by elementary
manipulations.
\end{proof}

\begin{Corollary}\label{cor:ex1}
We have the elliptic Karlsson--Minton type identity
\begin{gather*}
 \frac{\big(q,a^2 q;q,p\big)_n}{(aqz, aq/z;q,p)_n}
\frac{(bz, b/z;q,p)_s (dz, d/z;q,p)_{n-s}}
{(ab, b/a;q,p)_s (ad, d/a;q,p)_{n-s}} \\
\qquad{} ={}_{12}V_{11}\big(a^2, q^{-n}, az, a/z, aq/b, aq/d, abq^s, ad q^{n-s};q,p\big).
\end{gather*}
\end{Corollary}

\begin{proof}
We apply Theorem \ref{thm:id} to
\begin{gather*}
f(z)=\frac{(bz, b/z;q,p)_s (dz, d/z;q,p)_{n-s}}{(cz, c/z;q,p)_n}.\tag*{\qed}
\end{gather*}
\renewcommand{\qed}{}
\end{proof}

More generally, we have the following result.
\begin{Corollary}\label{cor:ex2}
We have the elliptic Karlsson--Minton type identity
\begin{gather*}
 \frac{\big(a^2 q, q;q,p\big)_n}{(aqz, aq/z;q,p)_n}
\prod_{j=1}^n \theta(b_j z, b_j/z;p)\notag\\
 \qquad{} =\sum_{k=0}^n q^{k(n+1)}\frac{\theta\big(a^2 q^{2k};p\big)}{\theta\big(a^2;p\big)}
\frac{\big(q^{-n}, a^2, az, a/z;q,p\big)_k}{\big(q, a^2 q^{n+1}, aqz, aq/z;q,p\big)_k}
\prod_{j=1}^n \theta\big(a b_k q^k, b_j q^{-k}/a;p\big).
\end{gather*}
\end{Corollary}

\begin{proof}
We take
\begin{gather}\label{idclear}
f(z)=\frac{\prod\limits_{j=1}^n \theta(b_j z, b_j/z;p)}{(cz,c/z;q,p)_n}
\end{gather}
and apply Theorem \ref{thm:id}.
\end{proof}

\begin{Remark}
It should be noted that if in the proof of Corollary~\ref{cor:ex2}
we instead would have taken
\begin{gather*}
f(z)=\frac{\prod\limits_{j=1}^t \theta(b_j z, b_j/z;p)}{(cz,c/z;q,p)_t}
\end{gather*}
for $0\le t\le n$, we would have just obtained the special
case of Corollary~\ref{cor:ex2} with $b_j\mapsto cq^{j-1}$ for $t+1\le n$,
which is clear carrying out those specializations in~\eqref{idclear}.
\end{Remark}

We now consider a multivariate version of Theorem~\ref{thm:id}.
Let us consider the space of functions
\begin{gather*}
W_{c_1,\dots, c_m}^{n_1,\dots,n_m}:=\operatorname{Span}_{\C}
\Bigg\{ \frac{g_{k_1,\dots,k_m}(z_1,\dots, z_m)}
{\prod\limits_{i=1}^m (c_i z_i, c_i/z_i;q,p)_{k_i}},
\, 0\le k_i\le n_i,\, i=1,\dots,m\Bigg\},
\end{gather*}
where $g_{k_1,\dots,k_m}(z_1,\dots, z_m)$
runs over all functions being holomorphic
in $z_1,z_2,\dots,z_m\neq 0$ and symmetric in $z_i$ and $1/z_i$, and
\begin{gather*}
g_{k_1,\dots,k_m}(z_1,\dots,pz_i,\dots,z_m)=
\frac 1{p^{k_i}z_i^{2k_i}}g_{k_1,\dots,k_m}(z_1,\dots,z_i,\dots,z_m),
\end{gather*}
for all $i=1,\dots,m$.

We def\/ine a multivariate extension of the elliptic Askey--Wilson
operator as follows.
\begin{gather*}
\mathcal{D}_{c_i,q,p;z_i}f(z_1,\dots, z_m)=
 2q^{\frac{1}{2}}z_i
\frac{\theta \big(c_i z_i q^{-\frac{1}{2}},c_i z_i q^{\frac{1}{2}},
c_i q^{-\frac{1}{2}}/z_i, c_i q^{\frac{1}{2}}/z_i;p\big)}{\theta\big(q,z_i ^2;p\big)}\\
\hphantom{\mathcal{D}_{c_i,q,p;z_i}f(z_1,\dots, z_m)=}{}
 \times\Big(f\big(z_1,\dots, q^{\frac{1}{2}}z_i,\dots, z_m\big)-
f\big(z_1,\dots, q^{-\frac{1}{2}}z_i,\dots, z_m\big) \Big),
\\
\mathcal{D}_{c_i,q,p;z_i}^{(k+1)}=\mathcal{D}_{c_i q^{\frac{3}{2}k},q,p;z_i}
\mathcal{D}_{c_i,q,p;z_i}^{(k)},
\end{gather*}
and for $\mathbf c=(c_1,\dots,c_m)$, $\mathbf k=(k_1,\dots,k_m)$,
and $\mathbf z=(z_1,\dots, z_m)$,
\begin{gather*}
\mathcal{D}_{\mathbf c, q,p;\mathbf z}^{(\mathbf k)}=
\mathcal{D}_{c_1,q,p;z_1}^{(k_1)}\cdots \mathcal{D}_{c_m,q,p;z_m}^{(k_m)}.
\end{gather*}

\begin{Theorem}\label{thm:multvTaylor}
If $f(z_1,\dots, z_m)$ is in $W_{\mathbf c}^{\mathbf n}$, then
\begin{gather}\label{eqn:multvTaylor}
f(z_1,\dots, z_m)=\sum_{k_1,\dots,k_m=0}^{n_1,\dots,n_m} f_{k_1,\dots, k_m}
\prod_{i=1}^m \frac{(a_i z_i, a_i/z_i ;q,p)_{k_i}}{(c_i z_i, c_i/z_i;q,p)_{k_i}},
\end{gather}
where
\begin{gather*}
f_{k_1,\dots, k_m}= \prod_{i=1}^m\frac{(-1)^{ k_i}
q^{-\frac{k_i(k_i-1)}{4}} (\theta(q;p))^{k_i}}
{(2 a_i)^{k_i}\big(q, c_i /a_i, a_i c_i q^{k_i -1};q,p\big)_{k_i}}
\big[ \mathcal{D}_{\mathbf c, q,p;\mathbf z}^{(\mathbf k)}
f(z_1,\dots, z_m)\big]_{z_i = a_i q^{k_i /2}},
\end{gather*}
where $\mathbf k=(k_1,\dots, k_m)$.
\end{Theorem}

\begin{proof}
Note that, for each $j=1,\dots, m$,
\begin{gather}\label{eqn:ndiff}
\mathcal{D}_{c_j,q,p;z_j}\prod_{i=1}^m \frac{(a_i z_i,a_i/z_i;q,p)_{n_i}}
{(c_i z_i, c_i/z_i;q,p)_{n_i}}\\
\quad{}
=\frac{(-1)2a_j \theta\big(c_j/a_j, a_j c_j q^{n_j -1}, q^{n_j};p\big)}{\theta(q;p)}
\frac{\big(a_j q^{\frac{1}{2}}z_j,a_j q^{\frac{1}{2}}/z_j;q,p\big)_{n_j -1}}
{\big(c_j q^{\frac{3}{2}}z_j, c_j q^{\frac{3}{2}}/z_j;q,p\big)_{n_j -1}}
\prod_{\substack{i=1,\\ i\ne j}}^m \frac{(a_i z_i,a_i/z_i;q,p)_{n_i}}
{(c_i z_i, c_i/z_i;q,p)_{n_i}}.\notag
\end{gather}
Iterating \eqref{eqn:ndiff} gives
\begin{gather*}
\big[ \mathcal{D}_{\mathbf c, q,p;\mathbf z}^{(\mathbf k)}
f(z_1,\dots, z_m)\big]_{z_i= a_i q^{k_i /2}}\\
\qquad{}
 = \prod_{i=1}^m (-1)^{k_i}(2a_i)^{k_i}q^{\frac{k_i (k_i -1)}{4}}
\frac{(q;q,p)_{n_i} \big(c_i/a_i, a_i c_i q^{n_i -1};q,p\big)_{k_i}}
{(q;q,p)_{n_i -k_i}\theta(q;p)^{k_i}}\\
 \qquad\quad {}\times \left[ \frac{\big(a_i q^{\frac{k_i}{2}}z_i ,
a_i q^{\frac{k_i}{2}}/z_i;q,p\big)_{n_i-k_i}}{\big(c_i q^{\frac{3k_i}{2}}z_i ,
c_i q^{\frac{3k_i}{2}}/z_i;q,p\big)_{n_i-k_i}}\right]_{z_i = a_i q^{\frac{k_i}{2}}}\\
\qquad{}= \prod_{i=1}^m (-1)^{k_i}(2a_i)^{k_i}q^{\frac{k_i (k_i -1)}{4}}
\frac{\big(q,c_i/a_i, a_i c_i q^{k_i -1};q,p\big)_{k_i}}{\theta(q;p)^{k_i}}\delta_{n_i k_i}.
\end{gather*}
Then the theorem follows by applying
$\mathcal{D}_{\mathbf c, q,p;\mathbf z}^{(\mathbf j)}$ to both sides of
\eqref{eqn:multvTaylor} and then setting $z_i=a_i q^{j_i/2}$, for
$i=1,\dots, m$ and $\mathbf j=(j_1,\dots, j_m)$.
\end{proof}

Now we provide a multivariate extension of Theorem~\ref{thm:Dn}.

\begin{Theorem}\label{thm:multvDn}
For $\mathbf n=(n_1,\dots, n_m)$, $\mathbf c=(c_1,\dots, c_m)$,
\begin{gather*}
\mathcal{D}_{\mathbf c,q,p;\mathbf z}^{(\mathbf n)}f(z_1,\dots, z_m) =
\prod_{i=1}^m\left[ (-2z_i)^{n_i}q^{\frac{n_i(3-n_i)}{4}}
\frac{\big(c_i q^{\frac{n_i}{2}-1}z_i, c_i q^{\frac{n_i}{2}-1}/z_i;q,p\big)_{n_i+1}}
{(\theta(q;p))^{n_i}}\right] \\
\qquad{}\times \sum_{k_m=0}^{n_m}\cdots\sum_{k_1=0}^{n_1}\prod_{i=1}^m
\Bigg( q^{k_i(n_i-k_i)}\begin{bmatrix}n_i\\k_i\end{bmatrix}_{p,q}
\frac{z_i^{2(k_i-n_i)}\big(c_i q^{\frac{n_i}{2}-k_i}z_i,
c_iq^{-\frac{n_i}{2}+k_i}/z_i;q,p\big)_{n_i-1}}{\big(q^{n_i-2k_i+1}z_i^2 ;q,p\big)_{k_i}
\big(q^{2k_i-n_i+1}z_i^{-2};q,p\big)_{n_i-k_i}} \\
\qquad\qquad{} \times f\big(q^{\frac{n_1}{2}-k_1}z_1,\dots,
q^{\frac{n_m}{2}-k_m}z_m\big)\Bigg).
\end{gather*}
\end{Theorem}

\begin{proof}
The theorem follows by applying Theorem~\ref{thm:Dn} successively for each
$i=1,\dots, m$.
\end{proof}

We combine Theorems~\ref{thm:multvTaylor} and \ref{thm:multvDn}
to obtain the following multivariable elliptic interpolation formula.

\begin{Theorem}\label{thm:multvid}
For $f(z_1,\dots, z_m)$ in $W_{\mathbf c}^{\mathbf n}$, we have
\begin{gather*}
\prod_{i=1}^m \frac{\big(a_i ^2 q, q, c_i z_i , c_i/z_i;q,p\big)_{n_i}}
{(a_i c_i,c_i/a_i, a_i qz_i, a_iq/z_i;q,p)_{n_i}} f(z_1,\dots, z_m)
 \\
\qquad {}=\sum_{k_1,\dots,k_m =0}^{n_1,\dots,n_m}\prod_{i=1}^m q^{k_i}
\frac{\theta\big(a_i^2 q^{2k_i};p\big)}
{\theta\big(a_i^2;p\big)} \frac{\big(q^{-n_i},a_i ^2, a_i q/c_i, a_i c_i q^{n_i}, a_i z_i ,
a_i/z_i;q,p\big)_{k_i}}{\big(q,a_i^2 q^{n_i+1}, a_i c_i, a_i q^{1-n_i}/c_i , a_i q z_i,
a_i q/z_i;q,p\big)_{k_i}} \\
\qquad\qquad{}\times
f\big(a_1 q^{k_1},\dots, a_m q^{k_m}\big).
\end{gather*}
\end{Theorem}
This theorem extends a result given by Ismail and
Stanton~\cite[Theorem~3.10]{IS2}, which can be obtained by taking
$m=2$, $p\to 0$, $c_1=c_2=0$ and $n_1=n_2=n$.

\begin{Corollary}\label{cor:multvex2}
We have the following multivariable elliptic
Karlsson--Minton type identity
\begin{gather*}
\prod_{i=1}^m \left(\frac{\big(a_i^2 q, q;q,p\big)_{n_i}}
{(a_iqz_i,a_i q/z_i;q,p)_{n_i}}
\prod_{j=1}^{s_i}\frac{(b_{ij}z_i,b_{ij}/z_i;q,p)_{v_{ij}}}
{(b_{ij}a_i,b_{ij}/a_i;q,p)_{v_{ij}}}\right)
\prod_{1\le i<j \le m}\Bigg(a_i^{w_{ij}}z_i^{-w_{ij}}
\theta(z_i z_j,z_i/z_j;p)^{w_{ij}}\\
\qquad{}\times
\prod_{l_{ij}=1}^{r_{ij}}
\frac{(\alpha_{l_{ij}} z_i z_j,\alpha_{l_{ij}} z_i/z_j,
\alpha_{l_{ij}} z_j/z_i,\alpha_{l_{ij}} /z_i z_j;q,p)_{u_{lij}}}
{(\alpha_{l_{ij}} a_i a_j ,\alpha_{l_{ij}} a_i/a_j,
\alpha_{l_{ij}} a_j/a_i ,\alpha_{l_{ij}}/a_i a_j;q,p)_{u_{lij}}}
\Bigg) \\
=\sum_{k_1,\dots, k_m=0}^{n_1,\dots,n_m}
\prod_{i=1}^m \Bigg(q^{k_i}
\frac{\theta\big(a_i^2 q^{2k_i};p\big)}{\theta\big(a_i ^2;p\big)}
\frac{\big(q^{-n_i}, a_i^2, a_i z_i, a_i/z_i;q,p\big)_{k_i}}
{\big(q,a_i^2 q^{n_i+1}, a_i q z_i,
a_i q/z_i;q,p\big)_{k_i}}\\
\qquad{}\times
\prod_{j=1}^{s_i}
\frac{\big(a_i b_{ij} q^{v_{ij}},a_i q/b_{ij};q,p\big)_{k_i}}
{\big(a_i b_{ij},a_i q^{1-v_{ij}}/b_{ij};q,p\big)_{k_i}}
\Bigg)\\
\qquad{} \times
\prod_{1\le i<j \le m}
\prod_{l_{ij}=1}^{r_{ij}}
q^{-2u_{l_{ij}}k_i}\\
\qquad{}\times
\frac{\big(\alpha_{l_{ij}} a_i a_j q^{u_{l_{ij}}},qa_ia_j/\alpha_{l_{ij}};q,p\big)_{k_i+k_j}
\big(\alpha_{l_{ij}} a_i q^{u_{l_{ij}}}/a_j,q a_i/a_j\alpha_{l_{ij}};q,p\big)_{k_i-k_j}}
{\big(\alpha_{l_{ij}} a_i a_j,q^{1-u_{l_{ij}}}a_ia_j/\alpha_{l_{ij}};q,p\big)_{k_i+k_j}
\big(\alpha_{l_{ij}} a_i/a_j,q^{1-u_{l_{ij}}}a_i/a_j\alpha_{l_{ij}};q,p\big)_{k_i-k_j}}
\\
\times
\prod_{1\le i<j \le m}q^{-w_{ij}k_i}
\theta\big(a_i a_j q^{k_i+k_j},a_iq^{k_i-k_j}/a_j;p\big)^{w_{ij}},
\end{gather*}
where
\begin{gather*}
n_i=\sum_{j=1}^{s_i} v_{ij}+
\sum_{j=1}^{i-1}w_{ji}+\sum_{j=i+1}^mw_{ij}
+2\sum_{l=1}^{r_{ij}}
\left(\sum_{j=1}^{i-1}u_{l_{ji}}+\sum_{j=i+1}^mu_{l_{ij}}\right),
\end{gather*}
for $i=1,\dots,m$.
\end{Corollary}

\begin{proof}
We apply Theorem \ref{thm:multvid} to
\begin{gather*}
f(z_1,\dots, z_m)= \prod_{i=1}^m
\frac{\prod\limits_{j=1}^{s_i} (b_{ij} z_i , b_{ij}/z_i;q,p)_{v_{ij}}}
{(c_i z_i,c_i/z_i;q,p)_{n_i}}
\prod_{1\le i<j \le m}z_i^{-w_{ij}}
\theta(z_i z_j,z_i/z_j;p)^{w_{ij}}\\
\hphantom{f(z_1,\dots, z_m)=}{} \times \prod_{1\le i<j \le m}
\prod_{l_{ij}=1}^{r_{ij}}(\alpha_{l_{ij}} z_i z_j,
\alpha_{l_{ij}} z_i/z_j,\alpha_{l_{ij}} z_j/z_i,
\alpha_{l_{ij}}/z_iz_j;q,p)_{u_{l_{ij}}},
\end{gather*}
where
\begin{gather*}
n_i=\sum_{j=1}^{s_i} v_{ij}+
\sum_{j=1}^{i-1}w_{ji}+\sum_{j=i+1}^mw_{ij}
+2\sum_{j=1}^{i-1}
\sum_{l_{ji}=1}^{r_{ji}}u_{l_{ji}}+2\sum_{j=i+1}^m
\sum_{l_{ij}=1}^{r_{ij}}u_{l_{ij}},
\end{gather*}
for $i=1,\dots,m$.
\end{proof}

Corollary~\ref{cor:multvex2} extends a result by Ismail and Stanton
(see \cite[Corollary 3.11]{IS2}), corresponding to a special case
of its $m=2$ instance.

More generally, $f(z_1,\dots,z_m)$ could involve symmetrized
products of $2^k$ factors
of the form
$(\lambda z_{i_1}^\pm z_{i_2}^\pm\cdots z_{i_k}^\pm;q,p)_y$
(the notation $z_{i_j}^\pm$ means that
the resprective variable could appear as~$z_{i_j}$ or~$z_{i_j}^{-1}$,
where all possible combinations appear),
where $\{i_1,\dots,i_k\}$ is any subset of $\{1,\dots,n\}$.
(In the corollary, we only considered factors for~$k=1,2$.)

Corollary \ref{cor:multvex2} can be easily seen to be
equivalent to its $u_{l_{ij}}=1$ and $v_{ij}= 1$ case,
for all~$i$,~$j$, in which case the respective factorials reduce
to simple theta functions. To recover the general case
from this special case one can suitably increase~$r_{ij}$ and
$s_1,\dots,s_m$ and choose the parameters partially
in geometric progression to obtain shifted factorials.
In particular, we can replace~$r_{ij}$ by $u_1+\dots+u_{r_{ij}}$
and relabel $\alpha_{u_1+\dots+u_{l_{ij}-1}+h}\mapsto
\alpha_{l_{ij}}q^{h-1}$,
for all $1\le l_{ij}\le r_{ij}$, $1\le h\le u_{l_{ij}}$, etc.
(One could even add extra bases, in addition to $q$.
This feature is typical for series of Karlsson--Minton type.)

For convenience, we restate the corollary in this equivalent form.

\begin{Corollary}\label{cor:multvex3}
We have the following multivariable elliptic Karlsson--Minton type identity
\begin{gather*}
\prod_{i=1}^m \left(\frac{(a_i^2 q, q;q,p)_{n_i}}
{(a_iqz_i,a_i q/z_i;q,p)_{n_i}}
\prod_{j=1}^{s_i}\frac{\theta(b_{ij}z_i,b_{ij}/z_i;p)}
{\theta(b_{ij}a_i,b_{ij}/a_i;p)}\right)\\
{}\times
\prod_{1\le i<j \le m}\left(a_i^{w_{ij}}z_i^{-w_{ij}}
\theta(z_i z_j,z_i/z_j;p)^{w_{ij}}
\prod_{l_{ij}=1}^{r_{ij}}
\frac{\theta(\alpha_{l_{ij}} z_i z_j,\alpha_{l_{ij}} z_i/z_j,
\alpha_{l_{ij}} z_j/z_i,\alpha_{l_{ij}} /z_i z_j;p)}
{\theta(\alpha_{l_{ij}} a_i a_j ,\alpha_{l_{ij}} a_i/a_j,
\alpha_{l_{ij}} a_j/a_i ,\alpha_{l_{ij}}/a_i a_j;p)}
\right)\\
=\sum_{k_1,\dots, k_m=0}^{n_1,\dots,n_m}
\prod_{i=1}^m \!\left(q^{k_i}
\frac{\theta(a_i^2 q^{2k_i};p)}{\theta(a_i ^2;p)}
\frac{(q^{-n_i}, a_i^2, a_i z_i, a_i/z_i;q,p)_{k_i}}
{(q,a_i^2 q^{n_i+1}, a_i q z_i,
a_i q/z_i;q,p)_{k_i}}
\prod_{j=1}^{s_i}\!
\frac{\theta(a_i b_{ij} q^{k_i},a_i q^{k_i}/b_{ij};p)}
{\theta(a_i b_{ij},a_i/b_{ij};p)}
\right)\\
{}\times
\prod_{1\le i<j \le m}
\prod_{l_{ij}=1}^{r_{ij}}
q^{-2k_i}
\frac{\theta\big(\alpha_{l_{ij}} a_i a_j q^{k_i+k_j},
q^{k_i+k_j}a_ia_j/\alpha_{l_{ij}},\alpha_{l_{ij}} a_i q^{k_i-k_j}/a_j,
q^{k_i-k_j} a_i/a_j\alpha_{l_{ij}};p\big)}
{\theta(\alpha_{l_{ij}} a_i a_j,a_ia_j/\alpha_{l_{ij}},
\alpha_{l_{ij}} a_i/a_j,a_i/a_j\alpha_{l_{ij}};p)}
\\
{}\times
\prod_{1\le i<j \le m}q^{-w_{ij}k_i}
\theta\big(a_i a_j q^{k_i+k_j},a_iq^{k_i-k_j}/a_j;p\big)^{w_{ij}},
\end{gather*}
where
\begin{gather*}
n_i=s_i+
\sum_{j=1}^{i-1}w_{ji}+\sum_{j=i+1}^mw_{ij}
+2\sum_{j=1}^{i-1}r_{ji}+2\sum_{j=i+1}^mr_{ij},
\end{gather*}
for $i=1,\dots,m$.
\end{Corollary}

\subsection{A quadratic elliptic Taylor expansion theorem}
In \cite{IS}
Ismail and Stanton also considered the basis
$\{\phi_k (z),\, 0\le k\le n\}$
where $\phi_k (z)=\big(q^{1/4}z, q^{1/4}/z;q^{1/2}\big)_k$. The set
\begin{gather*}
\left\{ \frac{\big(q^{1/4}z, q^{1/4}/z;q^{1/2},p\big)_k}
{(cz,c/z;q,p)_k},\, 0\le k\le n \right\}
\end{gather*}
apparently forms a basis for $W_c ^n$. We now provide a Taylor expansion
theorem with respect to this basis.

\begin{Theorem}\label{thm:Taylor2}
If $f$ is in $W_c ^n$, then
\begin{gather}
f(z)= \sum_{k=0}^{n} f_k \frac{\big(q^{1/4}z, q^{1/4}/z;q^{1/2},p\big)_k}
{(cz,c/z;q,p)_k},\label{eqn:Taylor2}
\end{gather}
where
\begin{gather*}
f_k = \frac{(-1)^k q^{-k/4}\theta (q;p)^k}
{2^k (q;q,p)_k \big(c q^{\frac{k}{2}-\frac{3}{4}};q^{1/2},p\big)_{2k}}
\big[\mathcal{D}_{c,q,p}^{(k)}f(z) \big]_{z=q^{1/4}}.
\end{gather*}
\end{Theorem}

\begin{proof}
Note that
\begin{gather*}
 \mathcal{D}_{c,q,p}^{(k)}\left( \frac{\big(q^{1/4}z, q^{1/4}/z;q^{1/2},p\big)_n}
{(cz,c/z;q,p)_n}\right)\notag\\
 \qquad{}=\frac{(-2)^k q^{k/4}\big(c q^{\frac{n}{2}-\frac{3}{4}};q^{1/2},p\big)_{2k}(q;q,p)_n}
{(q;q,p)_{n-k}\theta (q;p)^k}\frac{\big(q^{1/4}z, q^{1/4}/z;q^{1/2},p\big)_{n-k}}
{\big(cq^{\frac{3}{2}k}z,cq^{\frac{3}{2}k}/z;q,p\big)_{n-k}},
\end{gather*}
which can be proved by induction.
The theorem then follows by applying $\mathcal{D}_{c,q,p}^{(j)}$
to both sides of~(\ref{eqn:Taylor2}) and then setting $z=q^{1/4}$.
\end{proof}

In the following, we recover an elliptic quadratic summation by
Warnaar~\cite[Corollary 4.4; $b=a$]{Wa}, which was originally proved by using
inverse relations. Its $p=0$ case has been given earlier
by Gessel and Stanton~\cite[equation~(1.4)]{GS}.

\begin{Corollary}We have the following summation
\begin{gather*}
 \frac{(az, a/z;q,p)_n}{(cz, c/z;q,p)_n}\frac{\big(c q^{-1/4};q^{1/2},p\big)_{2n}}
{\big(a q^{-1/4};q^{1/2},p\big)_{2n}} \\
\qquad{}=\sum_{k=0}^n q^{\frac{k}{2}}\frac{\theta\big(c q^{\frac{3}{2}k-\frac{3}{4}};p\big)}
{\theta\big(c q^{-\frac{3}{4}};p\big)} \frac{\big(c/a, ac q^{n-1}, q^{-n};q,p\big)_k}
{(cz, c/z, q;q,p)_k}
\frac{\big(c q^{-\frac{3}{4}},q^{\frac{1}{4}}z, q^{\frac{1}{4}}/z;q^{1/2},p\big)_k}
{\big(aq^{-\frac{1}{4}},cq^{n-\frac{1}{4}},q^{\frac{3}{4}-n}/a;q^{1/2},p\big)_k}.
\end{gather*}
\end{Corollary}

\begin{proof}
We apply Theorem \ref{thm:Taylor2} to
\begin{gather*}
f(z)=\frac{(az, a/z;q,p)_n}{(cz, c/z;q,p)_n}.\tag*{\qed}
\end{gather*}
\renewcommand{\qed}{}
\end{proof}

\begin{Remark}
If we expand
\begin{gather*}
\frac{\big(q^{\frac{1}{4}}z, q^{\frac{1}{4}}/z;q^{\frac{1}{2}},p\big)_n}
{(cz,c/z;q,p)_n}
\end{gather*}
in terms of
\begin{gather*}
\frac{(az, a/z;q,p)_n}{(cz, c/z;q,p)_n}
\end{gather*}
using Theorem \ref{thm:Taylor}, we obtain
\begin{gather}
 \frac{\big(q^{\frac{1}{4}}z, q^{\frac{1}{4}}/z;q^{\frac{1}{2}},p\big)_n}
{\big(aq^{\frac{1}{4}}, q^{\frac{1}{4}}/a;q^{\frac{1}{2}},p\big)_n}
\frac{(ac, c/a;q,p)_n}{(cz, c/z;q,p)_n}\notag\\
 \qquad{}=\sum_{k=0}^n q^k \frac{\theta\big(acq^{2k-1};p\big)}{\theta\big(acq^{-1};p\big)}
\frac{\big(q^{-n},acq^{-1}, az, a/z, cq^{\frac{n}{2}-\frac{3}{4}},
cq^{\frac{n}{2}-\frac{1}{4}};q,p\big)_n}{\big(q, ac q^n , cz, c/z, aq^{\frac{1}{4}-\frac{n}{2}},
aq^{\frac{3}{4}-\frac{n}{2}};q,p\big)_n}.\label{eqn:remark1}
\end{gather}
At f\/irst glance this appears to be a true quadratic summation formula.
However, the right-hand side of~\eqref{eqn:remark1} is
\begin{gather*}
{}_{10}V_9\big(acq^{-1}; az, a/z, cq^{\frac{n}{2}-\frac{3}{4}},
cq^{\frac{n}{2}-\frac{1}{4}},q^{-n};q,p\big),
\end{gather*}
which, by Frenkel and Turaev's summation formula~\eqref{10V9},
can be reduced to
\begin{gather*}
\frac{\big(ac, c/a, q^{\frac{3}{4}-\frac{n}{2}}z,
q^{\frac{1}{4}-\frac{n}{2}}/z;q,p\big)_n}{\big(cz, c/z, aq^{\frac{3}{4}-\frac{n}{2}},
q^{\frac{3}{4}-\frac{n}{2}}/a;q,p\big)_n}.
\end{gather*}
Elementary manipulations can now be applied to transform this expression
to the left-hand side of \eqref{eqn:remark1}.
\end{Remark}

\section{Expansions involving cubic theta functions}

The cubic theta function $\gamma(z,a;p)$ with two independent
variables $z$ and $a$ in addition to the nome~$p$
was considered by S.~Bhargava~\cite{Bhar}.
(For a thorough treatment of the theory of cubic theta functions
in analogy to the theory of the classical Jacobi theta functions,
see~\cite{Schu}.) It is def\/ined by
\begin{gather}\label{def:cubictheta}
\gamma(z,a;p)=\sum_{k=-\infty}^{\infty}
\sum_{l=-\infty}^{\infty}p^{k^2 + kl+l^2}a^{k+l}z^{k-l}.
\end{gather}
This function, up to a normalization factor $\big(p^2;p^2\big)_\infty^2$
(independent from~$a$ and~$z$),
is almost equal to the following product of two modif\/ied
Jacobi theta functions
\begin{gather*}
\big(p^2;p^2\big)_\infty^2\theta\big({-}paz;p^2\big)\theta\big({-}pa/z;p^2\big)=
\sum_{k=-\infty}^{\infty}\sum_{l=-\infty}^{\infty}
p^{k^2+l^2}a^{k+l}z^{k-l},
\end{gather*}
which dif\/fers by the factor~$p^{kl}$ to
the summand of the double series in~\eqref{def:cubictheta}.
Because of this additional factor~$p^{kl}$, the cubic theta function
does not factorize into a product
of two modif\/ied Jacobi theta functions of such a simple form.
In principle though, the cubic theta function could be factorized
into two modif\/ied Jacobi theta functions, but their arguments
would have nontrivial expansions in $a$, $z$, and $p$.

From \eqref{def:cubictheta}, by replacing $(k,l)$ by $(l,k)$,
or $(k,l)$ by $(-l,-k)$, respectively, we immediately deduce
the symmetries~\cite{Bhar}
\begin{subequations}
\begin{gather}
\gamma(1/z,a;p) =\gamma(z,a;p),\label{eqn:condition1}\\
\intertext{and}
\gamma(z,1/a;p) =\gamma(z,a;p).\label{eqn:condition3}
\end{gather}
\end{subequations}
Further, from \eqref{def:cubictheta}, by replacing $(k,l)$ by
$(k+\lambda+\mu,l+\lambda)$, it is easy to verify that
for all inte\-gers~$\lambda$ and~$\mu$ the following functional equation holds~\cite{Bhar}:
\begin{gather*}
\gamma(z,a;p)=p^{3\lambda^2+3\lambda\mu+\mu^2}a^{2\lambda+\mu}z^\mu
\gamma\big(p^{\mu/2}z,p^{3(2\lambda+\mu)/2}a;p\big).
\end{gather*}
In particular, we have the quasi periodicities
\begin{subequations}
\begin{gather}
\gamma(pz,a;p) =\frac{1}{pz^2}\gamma(z,a;p),\label{eqn:condition2}\\
\intertext{and}
\gamma(z,p^3a;p) =\frac{1}{p^3a^2}\gamma(z,a;p).\label{eqn:condition4}
\end{gather}
\end{subequations}
Further, by separating the terms in the expansion of~$p$
according to whether the exponents of~$p$ are divisible by $3$ or not,
one can show~\cite{Bhar}
\begin{gather*}
\gamma(z,a;p)=\gamma\big(\sqrt{az^3},\sqrt{a^3/z^3};p^3\big)+
paz^{-1}\gamma\big(\sqrt{az^3},p^3\sqrt{a^3/z^3};p^3\big),
\end{gather*}
while separating the terms in the expansion of $z$
according to whether the exponents of $z$ are even or odd,
one has~\cite{BhFa}
\begin{gather*}
\gamma(z,a;p)=\big(p^6;p^6\big)_\infty\big(p^2;p^2\big)_\infty
\left[\theta\big({-}p^3a;p^6\big) \theta\big({-}pz^2;p^2\big)+
paz \theta\big({-}p^6a^2;p^6\big) \theta\big({-}p^2z^2;p^2\big)\right].
\end{gather*}

Cooper and Toh~\cite{CT} proved the following addition formulae which
will be useful in our computations.
\begin{Lemma}[\protect{\cite[Corollary~4.5]{CT}}]\label{lem:CT}
The following identities connecting modified Jacobi theta functions
and cubic theta functions hold:
\begin{subequations}
\begin{gather}
\gamma(z_1, \alpha;p)\theta(z_3/z_2 , z_2 z_3;p) -
\gamma(z_2 , \alpha;p)\theta(z_3/z_1 , z_1 z_3;p)\nonumber\\
\qquad{} =
\frac{z_3}{z_1}\gamma(z_3,\alpha;p)\theta(z_1/z_2,z_1 z_2;p),\label{eqn:CTlemma1}
\end{gather}
and
\begin{gather}
\gamma\big(z,\alpha_1;p^{\frac{1}{3}}\big)\theta(\alpha_3 / \alpha_2, \alpha_2 \alpha_3;p)
-\gamma\big(z,\alpha_2;p^{\frac{1}{3}}\big)\theta(\alpha_3 / \alpha_1, \alpha_1 \alpha_3;p)\nonumber\\
\qquad{}
=\frac{\alpha_3}{\alpha_1}\gamma\big(z,\alpha_3;p^{\frac{1}{3}}\big)
\theta(\alpha_1 / \alpha_2, \alpha_1 \alpha_2;p).\label{eqn:CTlemma2}
\end{gather}
\end{subequations}
\end{Lemma}

These two identities were proved in \cite{CT} by specializing
a $(3\times 3)$ determinant evaluation involving
cubic theta functions. They can also be proved directly,
expanding the cubic theta functions
and modif\/ied Jacobi theta functions as inf\/inite series,
together with clever series rearrangement.

Now we introduce the f\/irst cubic theta analogue of the
$q$-shifted factorial by
\begin{gather*}
\langle az, a/z;q,p\rangle_n :=
\prod_{j=0}^{n-1}\gamma\big(zq^{\frac{1-n}{2}+j},aq^{\frac{n-1}{2}};p\big).
\end{gather*}
From \eqref{eqn:condition2}
it is easy to see that the cubic shifted factorial satisf\/ies
\begin{gather*}
\langle apz, a/pz;q,p\rangle_n =\frac{1}{p^n z^{2n}}\langle az, a/z;q,p\rangle_n.
\end{gather*}
Together with \eqref{eqn:condition1}, this
implies that the quotient
\begin{gather*}
\frac{\langle az, a/z;q,p\rangle_n}{(cz, c/z;q,p)_n}
\end{gather*}
is in the space $W_c^n$.
Hence we can apply Theorem~\ref{thm:Taylor} to it, by which we obtain the
f\/irst cubic theta extension of Jackson's $_8\phi_7$ summation~\eqref{8phi7}.

\begin{Corollary}\label{cor:1c87}
We have the following summation
\begin{gather}
(bc, c/b;q,p)_n \frac{\langle az, a/z;q,p\rangle_n}{(cz, c/z;q,p)_n}
= \sum_{k=0}^n q^{nk}\left(\frac{c}{b}\right)^k \frac{\theta(bc q^{2k-1};p)}
{\theta(bcq^{-1};p)}
\frac{(q^{-n},bcq^{-1},bz, b/z;q,p)_k}
{(q,bcq^n, cz, c/z;q,p)_k}\notag\\
\qquad{} \times
\big\langle acq^{n-1},aq^{1-k}/c;q,p\big\rangle_k
\frac{\langle abq^k, aq^{-k}/b;q,p\rangle_n}
{\langle abq^n,aq^{-k}/b;q,p\rangle_k}.\label{1c87gl}
\end{gather}
\end{Corollary}

\begin{proof}
By using \eqref{eqn:CTlemma1} in Lemma~\ref{lem:CT}, we can prove
by induction that
\begin{gather*}
 \mathcal{D}_{c,q,p}^{(k)}\left( \frac{\langle az, a/z;q,p\rangle_n}
{(cz, c/z;q,p)_n}\right)\\
\qquad{}=
(2c)^k q^{\frac{3}{4}k(k-1)}\frac{(q^n;q^{-1},p)_k}
{\theta(q;p)^k}\prod_{j=0}^{k-1}\gamma(cq^{\frac{n-1}{2}+j},aq^{\frac{n-1}{2}};p)
\frac{\langle aq^{\frac{k}{2}}z, aq^{\frac{k}{2}}/z;q,p\rangle_{n-k}}
{(cq^{\frac{3}{2}k}z, cq^{\frac{3}{2}k}/z;q,p)_{n-k}}\\
\qquad{}=(2c)^k q^{\frac{3}{4}k(k-1)}\frac{(q^n;q^{-1},p)_k}
{\theta(q;p)^k}\langle acq^{n-1},aq^{1-k}/c;q,p\rangle_k
\frac{\langle aq^{\frac{k}{2}}z, aq^{\frac{k}{2}}/z;q,p\rangle_{n-k}}
{(cq^{\frac{3}{2}k}z, cq^{\frac{3}{2}k}/z;q,p)_{n-k}}.
\end{gather*}
Then the corollary follows from Theorem~\ref{thm:Taylor} while expanding
in the basis
\begin{gather*}
f(z)=\frac{(bz, b/z;q,p)_n}{(cz, c/z;q,p)_n}.\tag*{\qed}
\end{gather*}
\renewcommand{\qed}{}
\end{proof}

To recover Jackson's $_8\phi_7$ summation from Corollary~\ref{cor:1c87},
substitute
\begin{gather*}
a\mapsto\frac{-a}{p(1+a^2q^{n-1})}
\end{gather*}
in \eqref{1c87gl}, multiply both sides of the identity by
$(1+a^2q^{n-1})^n$ and let $p\to 0$.
When $p\to 0$, the usual theta shifted factorials clearly reduce to the
$q$-shifted factorials. That is, the quotient on the left-hand side
reduces to
\begin{gather*}
\lim_{p\to 0}\frac{(bc,c/b;q,p)_n}{(cz,c/z;q,p)_n}=
\frac{(bc,c/b;q)_n}{(cz,c/z;q)_n}.
\end{gather*}
What happens with the cubic theta shifted factorial?
We have
\begin{gather*}
 \lim_{p\to 0}\big(1+a^2q^{n-1}\big)^n
\Big\langle\frac{-az}{p(1+a^2q^{n-1})},
\frac{-a}{p(1+a^2q^{n-1})z};p\Big\rangle_n\\
 =
\big(1+a^2q^{n-1}\big)^n\lim_{p\to 0}\prod_{j=0}^{n-1}
\gamma \left(zq^{\frac{1-n}2+j},
\frac{-aq^{\frac{n-1}2}}{p(1+a^2q^{n-1})};p\right)\\
 =
\big(1+a^2q^{n-1}\big)^n\prod_{j=0}^{n-1}\lim_{p\to 0}
\sum_{k=-\infty}^\infty\sum_{l=-\infty}^\infty
(-1)^{k+l}p^{k^2+kl+l^2-k-l}\left(\frac{aq^{\frac{n-1}2}}
{1+a^2q^{n-1}}\right)^{k+l}\big(zq^{\frac{1-n}2+j}\big)^{k-l}.
\end{gather*}
Now it is easy to see that for $p\to 0$ only three
terms in the various double inf\/inite series survive.
These three terms correspond to the cases $(k,l)=(0,0),(1,0),(0,1)$.
The last expression thus reduces to{\samepage
\begin{gather*}
\big(1+a^2q^{n-1}\big)^n\prod_{j=0}^{n-1}\left(1-\frac{aq^{\frac{n-1}2}}
{1+a^2q^{n-1}}\big(zq^{\frac{1-n}2+j}+z^{-1}q^{\frac{n-1}2-j}\big)\right)\\
\qquad{} =\prod_{j=0}^{n-1}\left(1+a^2q^{n-1}-aq^{\frac{n-1}2}
\big(zq^{\frac{1-n}2+j}+z^{-1}q^{\frac{n-1}2-j}\big)\right)\\
\qquad{} =
\prod_{j=0}^{n-1}(1-azq^j)\big(1-aq^{n-1-j}/z\big)=
(az,a/z;q)_n.
\end{gather*}
We take similar limits on the right-hand side
of \eqref{1c87gl}.}

Our next result involves elliptic interpolation of cubic theta shifted
factorials.
\begin{Corollary}
We have the following Karlsson--Minton type identity involving
cubic theta functions
\begin{gather*}
 \frac{\big(a^2 q,q;q,p\big)_n}{(aqz,aq/z;q,p)_n}\langle bz,b/z;q,p\rangle_n\notag\\
\qquad{}= \sum_{k=0}^n q^{k(n+1)}\frac{\theta\big(a^2 q^{2k};p\big)}
{\theta\big(a^2;p\big)}\frac{\big(q^{-n},a^2,az,a/z;q,p\big)_k}{\big(q,a^2 q^{n+1},aqz,aq/z;q,p\big)_k}
\big\langle abq^k, bq^{-k}/a;q,p\big\rangle_n.
\end{gather*}
\end{Corollary}

\begin{proof}
We apply Theorem \ref{thm:id} to
\begin{gather*}
f(z)=\frac{\langle bz, b/z;q,p\rangle_n}{(cz, c/z;q,p)_n}.\tag*{\qed}
\end{gather*}
\renewcommand{\qed}{}
\end{proof}

More generally, we have the following Karlsson--Minton type
identity involving cubic theta functions.
\begin{Corollary}
We have
\begin{gather*}
 \frac{\big(a^2 q, q;q,p\big)_n}{(aqz, aq/z;q,p)_n}\prod_{i=1}^s
\theta(b_i z, b_i/z;p)\prod_{j=1}^{n-s}\gamma(z,d_j;p)
 =\sum_{k=0}^n q^{k(n+1)}\frac{\theta(a^2 q^{2k};p)}{\theta(a^2;p)}\\
 \qquad{}\times
\frac{\big(q^{-n},a^2,az,a/z;q,p\big)_k}{\big(q, a^2 q^{n+1},aqz, aq/z;q,p\big)_k}
\prod_{i=1}^s \theta\big(ab_i q^k , b_i q^{-k}/a;p\big)
\prod_{j=1}^{n-s}\gamma\big(aq^k, d_j;p\big).
\end{gather*}
\end{Corollary}

\begin{proof}
We apply Theorem \ref{thm:id} to
\begin{gather*}
f(z)=\frac{\prod\limits_{i=1}^s \theta(b_i z, b_i/z;p)
\prod\limits_{j=1}^{n-s}\gamma(z,d_j;p)}{(cz, c/z;q,p)_n}.\tag*{\qed}
\end{gather*}
\renewcommand{\qed}{}
\end{proof}

Our next result concerns a cubic theta extension of
Gessel and Stanton's quadratic summation \cite[equation~(1.4)]{GS}.

\begin{Corollary}\label{cor:cgs}
We have the following summation
\begin{gather}
 \frac{\langle az, a/z;q,p\rangle_n}{(cz, c/z;q,p)_n}
\big(cq^{-\frac{1}{4}}, cq^{\frac{1}{4}};q,p\big)_n\notag\\
 \qquad{} =\sum_{k=0}^n c^k q^{\frac{k}{4}(k-2)+nk}
\frac{\theta\big(cq^{\frac{3}{2}k-\frac{3}{4}};p\big)}
{\theta\big(cq^{\frac{k}{2}-\frac{3}{4}};p\big)}\frac{\big(q^{-n};q,p\big)_k}{(q;q,p)_k}
\frac{\big(c q^{-\frac{1}{4}};q^{\frac{1}{2}},p\big)_k}
{\big(cq^{n-\frac{1}{4}};q^{\frac{1}{2}},p\big)_k}
\frac{\big(q^{\frac{1}{4}}z, q^{\frac{1}{4}}/z;q^{\frac{1}{2}},p\big)_k}
{(cz, c/z;q,p)_k}\notag\\
 \qquad\quad{}\times \big\langle ac q^{n-1},aq^{1-k}/c;q,p\big\rangle_k
 \big\langle aq^{\frac{k}{2}+\frac{1}{4}},
aq^{\frac{k}{2}-\frac{1}{4}};q,p\big\rangle_{n-k}.\label{cor:cgsgl}
\end{gather}
\end{Corollary}

\begin{proof}
We apply Theorem \ref{thm:Taylor2} to
\begin{gather*}
f(z)=\frac{\langle az, a/z;q,p\rangle_n}{(cz, c/z;q,p)_n}.\tag*{\qed}
\end{gather*}
\renewcommand{\qed}{}
\end{proof}

Similarly to the way we recovered Jackson's $_8\phi_7$
summation from Corollary~\ref{cor:1c87},
Gessel and Stanton's quadratic summation can be readily
obtained by substituting $a\mapsto -p^{-1}a/\big(1+a^2q^{n-1}\big)$
in \eqref{cor:cgsgl},
multiplying both sides by $(1+a^2q^{n-1})^n$
and taking the limit $p\to0$.

Next, we def\/ine the second cubic theta shifted factorial, with base $p^{1/3}$:
\begin{gather*}
\langle\langle az, a/z;q,p^{\frac{1}{3}}\rangle\rangle_n := \prod_{j=0}^{n-1}
\gamma\big(aq^{\frac{n-1}{2}},zq^{\frac{1-n}{2}+j};p^{\frac{1}{3}}\big).
\end{gather*}

Recalling equations~\eqref{eqn:condition2} and~\eqref{eqn:condition4}
(which we reformulate after interchanging~$a$ and~$z$),
\begin{gather*}
\gamma(a, z;p) =\gamma(a,1/z;p),\qquad
\gamma(a,z;p) =p^3 z^2 \gamma\big(a,p^3 z;p\big),
\end{gather*}
we see that
\begin{gather*}
\big\langle\big\langle apz, a/pz;q,p^{\frac{1}{3}}\big\rangle\big\rangle_n =
\frac{1}{p^n z^{2n}}\big\langle\big\langle az, a/z;q,p^{\frac{1}{3}}\big\rangle\big\rangle_n.
\end{gather*}
This implies that the quotient
\begin{gather*}
\frac{\big\langle\big\langle az, a/z;q,p^{\frac{1}{3}}\big\rangle\big\rangle_n}{(cz, c/z;q,p)_n}
\end{gather*}
is also in the space $W_c^n$. Thus, Theorem~\ref{thm:Taylor}
can be applied to it, by which we obtain the
second cubic theta extension of Jackson's $_8\phi_7$ summation \eqref{8phi7}.

\begin{Corollary}\label{cor:2c87}
We have the following summation
\begin{gather}
 \frac{\big\langle\big\langle bz, b/z;q,p^{\frac{1}{3}}\big\rangle\big\rangle_n}
{(cz, c/z;q,p)_n}(ac, c/a;q,p)_n
 =\sum_{k=0}^n q^{nk}\left( \frac{c}{a}\right)^k
\frac{\theta\big(ac q^{2k-1};p\big)}{\theta\big(ac q^{-1};p\big)}
\frac{\big(q^{-n},acq^{-1},az, a/z;q,p\big)_k}{\big(q, acq^n , cz, c/z;q,p\big)_k}
\notag\\ \qquad\quad{}\times
\big\langle\big\langle bc q^{n-1}, bq^{1-k}/c;q,p^{\frac{1}{3}}\big\rangle\big\rangle_k
\big\langle\big\langle abq^k,b/a;q,p^{\frac{1}{3}}\big\rangle\big\rangle_{n-k}.\label{2c87gl}
\end{gather}
\end{Corollary}

\begin{proof}
Note that by using \eqref{eqn:CTlemma2} in Lemma \ref{lem:CT},
we can show by induction that
\begin{gather*}
 \mathcal{D}_{c,q,p}^{(k)}\left( \frac{\big\langle\big\langle bz, b/z;
q,p^{\frac{1}{3}}\big\rangle\big\rangle_n}
{(cz, c/z;q,p)_n}\right)\\
 \qquad{}=(2c)^k q^{\frac{3}{4}k(k-1)}\frac{\big(q^n;q^{-1},p\big)_k}
{\theta(q;p)^k}\prod_{j=0}^{k-1}\gamma\big(bq^{\frac{n-1}{2}},
cq^{\frac{n-1}{2}+j};p^{\frac{1}{3}}\big)\frac{\big\langle\big\langle bq^{\frac{k}{2}}z,
bq^{\frac{k}{2}}/z;q,p^{\frac{1}{3}}\big\rangle\big\rangle_{n-k}}{\big(cq^{\frac{3}{2}k}z,
cq^{\frac{3}{2}k}/z;q,p\big)_{n-k}}\\
\qquad{} =(2c)^k q^{\frac{3}{4}k(k-1)}\frac{\big(q^n;q^{-1},p\big)_k}
{\theta(q;p)^k}\big\langle\big\langle bc q^{n-1}, bq^{1-k}/c;
q,p^{\frac{1}{3}}\big\rangle\big\rangle_k
\frac{\big\langle\big\langle bq^{\frac{k}{2}}z, bq^{\frac{k}{2}}/z;
q,p^{\frac{1}{3}}\big\rangle\big\rangle_{n-k}}
{\big(cq^{\frac{3}{2}k}z, cq^{\frac{3}{2}k}/z;q,p\big)_{n-k}}.
\end{gather*}
Using this, we apply Theorem \ref{thm:Taylor} to
\begin{gather*}
f(z)= \frac{\big\langle\big\langle bz, b/z;q,p^{\frac{1}{3}}\big\rangle\big\rangle_n}
{(cz, c/z;q,p)_n}.\tag*{\qed}
\end{gather*}
\renewcommand{\qed}{}
\end{proof}

To recover Jackson's $_8\phi_7$ summation from Corollary~\ref{cor:2c87},
substitute
\begin{gather*}
b\mapsto\frac{-b}{p^{\frac 13}\big(1+b^2q^{n-1}\big)}
\end{gather*}
in \eqref{2c87gl}, multiply both sides of the identity by
$\big(1+b^2q^{n-1}\big)^n$ and let $p\to 0$.
When $p\to 0$, the usual theta shifted factorials reduce to the
$q$-shifted factorials and
the cubic theta shifted factorial on the left-hand side of \eqref{2c87gl}
becomes
\begin{gather*}
 \lim_{p\to 0}\big(1+b^2q^{n-1}\big)^n\Big\langle \Big\langle\frac{-bz}{p^{\frac 13}
\big(1+b^2q^{n-1}\big)},
\frac{-b}{p^{\frac 13}\big(1+b^2q^{n-1}\big)z};p^{\frac 13}
\Big\rangle \Big\rangle_n\\
 \qquad{}=
\big(1+b^2q^{n-1}\big)^n\lim_{p\to 0}\prod_{j=0}^{n-1}\gamma \left(
\frac{-bq^{\frac{n-1}2}}{p^{\frac 13}(1+b^2q^{n-1})},
zq^{\frac{1-n}2+j};p^{\frac 13}\right)\\
\qquad{} =
\big(1+b^2q^{n-1}\big)^n\prod_{j=0}^{n-1}\lim_{p\to 0}
\sum_{k=-\infty}^\infty\sum_{l=-\infty}^\infty
(-1)^{k-l}p^{\frac 13(k^2+kl+l^2-k+l)}\\
\qquad\quad{}\times \big(zq^{\frac{1-n}2+j}\big)^{k+l}
\left(\frac{bq^{\frac{n-1}2}}
{1+b^2q^{n-1}}\right)^{k-l}.
\end{gather*}
Now it is easy to see that for $p\to 0$ only three
terms in the various double inf\/inite series survive.
These correspond to the cases $(k,l)=(0,0),(1,0),(0,-1)$.
The last expression thus reduces to
\begin{gather*}
\big(1+b^2q^{n-1}\big)^n\prod_{j=0}^{n-1}\left(1-\frac{bq^{\frac{n-1}2}}
{1+b^2q^{n-1}}\big(zq^{\frac{1-n}2+j}+z^{-1}q^{\frac{n-1}2-j}\big)\right)\\
\qquad {} =\prod_{j=0}^{n-1}\left(1+b^2q^{n-1}-bq^{\frac{n-1}2}
\big(zq^{\frac{1-n}2+j}+z^{-1}q^{\frac{n-1}2-j}\big)\right)\\
\qquad{} =
\prod_{j=0}^{n-1}\big(1-bzq^j\big)\big(1-bq^{n-1-j}/z\big)=
(bz,b/z;q)_n.
\end{gather*}
We take similar limits on the right-hand side
of equation~\eqref{2c87gl}.

Our f\/inal result concerns another cubic theta extension of
Gessel and Stanton's quadratic summation \cite[equation~(1.4)]{GS}.

\begin{Corollary}\label{cor:cgs2}
We have the following summation
\begin{gather*}
 \frac{\big\langle\big\langle az, a/z;q,p^{\frac 13}\big\rangle\big\rangle_n}{(cz, c/z;q,p)_n}
\big(cq^{-\frac{1}{4}}, cq^{\frac{1}{4}};q,p\big)_n \\
\qquad{} =\sum_{k=0}^n c^k q^{\frac{k}{4}(k-2)+nk}
\frac{\theta\big(cq^{\frac{3}{2}k-\frac{3}{4}};p\big)}
{\theta\big(cq^{\frac{k}{2}-\frac{3}{4}};p\big)}\frac{\big(q^{-n};q,p\big)_k}{(q;q,p)_k}
\frac{\big(c q^{-\frac{1}{4}};q^{\frac{1}{2}},p\big)_k}
{\big(cq^{n-\frac{1}{4}};q^{\frac{1}{2}},p\big)_k}
\frac{\big(q^{\frac{1}{4}}z, q^{\frac{1}{4}}/z;q^{\frac{1}{2}},p\big)_k}
{(cz, c/z;q,p)_k} \\
 \qquad\quad{}\times \big\langle\big\langle ac q^{n-1},aq^{1-k}/c;
q,p^{\frac 13}\big\rangle\big\rangle_k
\big\langle\big\langle aq^{\frac{k}{2}+\frac{1}{4}},
aq^{\frac{k}{2}-\frac{1}{4}};q,p^{\frac 13}\big\rangle\big\rangle_{n-k}.
\end{gather*}
\end{Corollary}

\begin{proof}
We apply Theorem~\ref{thm:Taylor2} to
\begin{gather*}
f(z)=\frac{\big\langle\big\langle az, a/z;q,p^{\frac 13}\big\rangle\big\rangle_n}
{(cz, c/z;q,p)_n}.\tag*{\qed}
\end{gather*}
 \renewcommand{\qed}{}
\end{proof}

\subsection*{Acknowledgements}
The work in this paper has been supported by FWF Austrian Science Fund
grant F50-08 within the SFB ``Algorithmic and enumerative combinatorics''.

\pdfbookmark[1]{References}{ref}
\LastPageEnding

\end{document}